\documentclass[a4paper]{article}
\usepackage[centertags, leqno]{amsmath}
\usepackage{amsthm}
\usepackage{amssymb}
\usepackage{amsfonts}
\usepackage{mathrsfs}
\usepackage{dsfont}
\usepackage{lineno}
\usepackage{geometry}
\theoremstyle{plain}
\newtheorem{theorem}[equation]{Theorem}
\newtheorem{corollary}[equation]{Corollary}
\newtheorem{lemma}[equation]{Lemma}
\newtheorem{proposition}[equation]{Proposition}
\theoremstyle{definition}
\newtheorem{defi}[equation]{Definition}

\newtheorem{remark}[equation]{Remark}

\newcommand{\mydia}{\hfill $\Diamond$}

\newcommand{\assPro}[1]{% arg1: Symbol of the process
	\ensuremath{\langle #1 \rangle}}

\newcommand{\ind}{\text{\large{$\mathds{1}$}}}

\newcommand{\myd}{\mbox{\upshape d}}

\long\def\symbolfootnote[#1]#2{\begingroup%
\def\thefootnote{\fnsymbol{footnote}}\footnote[#1]{#2}\endgroup}

\numberwithin{equation}{section}

\title{A Note on One-dimensional Stochastic Differential Equations with Generalized Drift\thanks{Work supported in part by the European Community's FP 7 Programme under contract PITN-GA-2008-213841, Marie Curie ITN "Controlled Systems".}}

\author{Stefan Blei \\ stefan.blei@uni-jena.de \and Hans-J\"urgen Engelbert \\ hans-juergen.engelbert@uni-jena.de}

\date{Friedrich-Schiller-Universit\"at Jena, \\ Fakult\"at f\"ur Mathematik und Informatik, \\ Institut f\"ur Stochastik, \\ D-07743 Jena, Germany \\[2ex] \today}

\begin{document}
\maketitle
\hrule 
\begin{abstract}
\noindent We consider one-dimensional stochastic differential equations with generalized drift which involve the local time $L^X$ of the solution process:
\[
	X_t = X_0 + \int_0^t b(X_s) \, \myd B_s + \int_\mathbb{R} L^X(t,y) \, \nu(\myd y)\,,
\]
where $b$ is a measurable real function, $B$ is a Wiener process and $\nu$ denotes a set function which is defined on the bounded Borel sets of the real line $\mathbb{R}$ such that it is a finite signed measure on $\mathscr{B}([-N,N])$ for every $N \in \mathbb{N}$. This kind of equation is, in dependence of using the right, the left or the symmetric local time, usually studied under the atom condition $\nu(\{x\}) < 1/2$, $\nu(\{x\}) > -1/2$ and $|\nu(\{x\})| < 1$, respectively. This condition allows to reduce an equation with generalized drift to an equation without drift and to derive conditions on existence and uniqueness of solutions from results for equations without drift. The main aim of the present note is to treat the cases $\nu(\{x\}) \geq 1/2$, $\nu(\{x\}) \leq -1/2$ and $|\nu(\{x\})| \geq 1$, respectively, for some $x \in \mathbb{R}$, and we give a complete description of the features of equations with generalized drift and their solutions in these cases.
\\[2ex]
\end{abstract}
\hrule
\section{Introduction and Basic Definitions}
Let be $b$ a measurable real function and $\nu$ a set function which is defined on the bounded Borel sets of the real line $\mathbb{R}$ such that it is a finite signed measure on $\mathscr{B}([-N,N])$ for every $N \in \mathbb{N}$. In the present note, we deal with the one-dimensional stochastic differential equation (SDE) with so-called \emph{generalized drift} introduced as
\begin{equation}\label{eqn:SDE_mvd}
	X_t = X_0 + \int_0^t b(X_s) \, \myd B_s + \int_\mathbb{R} L^X(t,y) \, \nu(\myd y)\,,
\end{equation}
where $B$ is a Wiener process and $L^X$ denotes either the right, the left or the symmetric local time of the unknown process $X$. We call $\nu$ appearing in Eq. (\ref{eqn:SDE_mvd}) \emph{drift measure.} \\
\indent SDEs of type (\ref{eqn:SDE_mvd}) with generalized drift have been studied previously by many authors. We refer the reader to Harrison and Shepp \cite{harrison_shepp}, N.I. Portenko \cite{portenko}, D.W. Stroock and M. Yor \cite{stroock_yor} and J.F. Le Gall \cite{LeGall_1983}, \cite{LeGall_1984}. H.J. Engelbert and W. Schmidt \cite{engelbert_schmidt:1985}, \cite{engelbert_schmidt:1989_III} derived rather weak necessary and sufficient conditions on existence and uniqueness of solutions to SDEs with generalized drift. More recently, R.F. Bass and Z.-Q. Chen \cite{bass_chen} also considered SDEs of type (\ref{eqn:SDE_mvd}). \\
\indent To treat the general equation (\ref{eqn:SDE_mvd}), in case of considering the right (resp., left, symmetric) local time the additional assumption 
\begin{equation}\label{eqn:cond_atoms}
	\nu(\{x\}) < \frac{1}{2} \quad (\text{resp., } \nu(\{x\}) > -\frac{1}{2}, \ |\nu(\{x\})| < 1), \qquad x \in \mathbb{R}\,,
\end{equation}
is posed on $\nu$ (cf. \cite{engelbert_schmidt:1985},\cite{engelbert_schmidt:1989_III},\cite{LeGall_1984},\cite{stroock_yor}). This condition allows to reduce Eq. (\ref{eqn:SDE_mvd}) to an equation without drift, i.e., an equation of type (\ref{eqn:SDE_mvd}) where the drift measure is the zero measure. Therefore, well-known conditions on existence and uniqueness of solutions to equations without drift can be used to derive conditions on existence and uniqueness of solutions to Eq. (\ref{eqn:SDE_mvd}) (see e.g. \cite{engelbert_schmidt:1985},\cite{engelbert_schmidt:1989_III}). More precisely, under condition (\ref{eqn:cond_atoms}) the integral equation
\begin{equation}\label{eqn:integral_eqn}
	g(x) = \left\{
					\begin{array}{ll}
						\displaystyle 1 - 2 \int_{[0,x]} F(g,y) \, \nu(\myd y), & x \geq 0, \medskip \\
						\displaystyle 1 + 2 \int_{(x,0)} F(g,y) \, \nu(\myd y), & x \geq 0,
					\end{array}
				 \right.
\end{equation}
where
\[
	F(g,x) = g(x-) \quad(\text{resp., } g(x), \ \left(g(x) + g(x-)\right)/2), \qquad x \in \mathbb{R},
\]
admits a unique c\`{a}dl\`{a}g solution. This solution $g$ is strictly positive and the strictly increasing and continuous primitive $G(x) = \int_0^x g(y)\, \myd y$ transforms Eq. (\ref{eqn:SDE_mvd}) into an equation without drift (cf. \cite{engelbert_schmidt:1985}, Proposition 1, or \cite{engelbert_schmidt:1989_III}, Proposition (4.29)). In the several cases, the explicit form of the solution to (\ref{eqn:integral_eqn}) can be found in \cite{engelbert_schmidt:1989_III}, (4.26), (4.26$^\prime$) and (4.26$^{\prime\prime}$), respectively. \\
\indent The case $\nu(\{x\}) = 1/2$ (resp., $\nu(\{x\}) = -1/2$, $|\nu(\{x\})| = 1$) for some $x \in \mathbb{R}$ is excluded in (\ref{eqn:integral_eqn}) since it corresponds, as we will see, to a reflecting barrier at the point $x$, which requires different methods to treat Eq. (\ref{eqn:SDE_mvd}) than by assuming (\ref{eqn:cond_atoms}) (cf. W. Schmidt \cite{schmidt:1989}, R.F. Bass and Z.-Q. Chen \cite{bass_chen}). Moreover, in the case that $\nu(\{x\}) > 1/2$ (resp., $\nu(\{x\}) < -1/2$, $|\nu(\{x\})| > 1$) holds for some $x \in \mathbb{R}$, in general, there is no solution to Eq. (\ref{eqn:SDE_mvd}). \\
\indent Indeed, in their famous paper on skew Brownian motion, J.M. Harrison and L.A. Shepp \cite{harrison_shepp} studied Eq. (\ref{eqn:SDE_mvd}) with symmetric local time in the special case $X_0 = 0$, $b \equiv 1$ and $\nu = \beta \, \delta_0$, where $\delta_0$ is the Dirac measure in zero, and they proved that there is no solution for the case $|\beta|>1$. \\
\indent Referring to \cite{harrison_shepp}, J.F. Le Gall \cite{LeGall_1984} (see after the proof of \cite{LeGall_1984}, Theorem 2.3) asserted, without giving a proof, that in his context ($b$ is of finite variation and bounded from below by a strictly positive constant) the result on the non-existence of a solution started at $x_0$ can be extended to Eq. (\ref{eqn:SDE_mvd}) with symmetric local time if $|\nu(\{x_0\})| > 1$. \\
\indent Also for symmetric local time, R.F. Bass and Z.-Q. Chen \cite{bass_chen} stated some propositions if $|\nu(\{x\})| > 1$ or $|\nu(\{x\})| = 1$ for some $x \in \mathbb{R}$ under the assumption, besides others, that the diffusion coefficient $b$ is bounded below by a positive constant (see \cite{bass_chen}, Theorem 3.2 and 3.3). However, in both formulations and proofs, we feel that there is not enough clarity which had enabled us to follow their arguments. In particular, they have not pointed out where their assumptions on $b$ are used. But we shall see below that for general diffusion coefficient $b$ their Theorem 3.2 does not remain true. \\
\indent The purpose of the present note is to give a general, complete and rigorous approach to Eq. (\ref{eqn:SDE_mvd}) under the condition that, contrary to (\ref{eqn:cond_atoms}), for some $x \in \mathbb{R}$ the drift measure $\nu$ satisfies
\[
	\nu(\{x\}) > 1/2  \quad (\text{resp., } \nu(\{x\}) < -1/2, \ |\nu(\{x\})| > 1)
\]
or
\[
  \nu(\{x\}) = 1/2 \quad (\text{resp., } \nu(\{x\}) = -1/2, \ |\nu(\{x\})| = 1).
\]
The basic idea is to provide an insight into the behaviour of the local times $L^X(t,x)$ of solutions $X$ of Eq. (\ref{eqn:SDE_mvd}) in such points $x \in \mathbb{R}$ violating (\ref{eqn:cond_atoms}). This gives rise for an application of Tanaka's formula, followed by a space transformation, to conclude full information about the features of the solution. \\
\indent Throughout the paper, $(\Omega, \mathcal{F}, \mathbf{P})$ stands for a complete probability space endowed with a filtration $\mathbb{F} = (\mathcal{F}_t)_{t \geq 0}$ which satisfies the usual conditions, i.e., $\mathbb{F}$ is right-continuous and $\mathcal{F}_0$ contains all sets from $\mathcal{F}$ which have $\mathbf{P}$-measure zero. For a process $X = (X_t)_{t \geq 0}$ the notation $(X,\mathbb{F})$ indicates that $X$ is $\mathbb{F}$-adapted. The processes considered in the following belong to the class of continuous semimartingales up to a stopping time $S$ and local times of such processes will be an important tool. Given an $\mathbb{F}$-stopping time $S$, we say that $(X,\mathbb{F})$ is a \emph{semimartingale up to $S$} if there exists an increasing sequence $(S_n)_{n \in \mathbb{N}}$ of $\mathbb{F}$-stopping times such that $S = \lim_{n\rightarrow +\infty} S_n$ and the process $(X^{S_n},\mathbb{F})$ obtained by stopping $(X,\mathbb{F})$ in $S_n$ is a real-valued semimartingale for every $n\in\mathbb{N}$. Analogously, we introduce the notion of a \emph{local martingale up to $S$}. \\
\indent If $(X,\mathbb{F})$ is a semimartingale up to $S$, then we can find a decomposition
\begin{equation}\label{eqn:semi_decomposition}
	X_t = X_0 + M_t + V_t\,, \qquad t < S, \ \mathbf{P}\text{-a.s.},
\end{equation}
on $[0,S)$, where $(M,\mathbb{F})$ is a local martingale up to $S$ with $M_0 = 0$ and $(V, \mathbb{F})$ is a right-continuous process whose paths are of bounded variation on $[0,t]$ for every $t < S$ and with $V_0=0$. If $X$ is continuous on $[0,S)$, then there exists a decomposition such that $M$ and $V$ are continuous on $[0,S)$ and this decomposition is unique on $[0,S)$. For any continuous local martingale $(M, \mathbb{F})$ up to $S$ by $\assPro{M}$ we denote the continuous increasing process, which is uniquely determined on $[0,S)$, such that $(M^2 - \assPro{M}, \mathbb{F})$ is a continuous local martingale up to $S$ and $\assPro{M}_0 = 0$. For a continuous semimartingale $(X,\mathbb{F})$ up to $S$ we set $\assPro{X} = \assPro{M}$, where $M$ is the continuous local martingale up to $S$ in the decomposition (\ref{eqn:semi_decomposition}) of $X$. \\
\indent We now recall some facts which are well-known for continuous semimartingales (see for example \cite{revuzyor}, Ch. VI, \S{}1). Their extension to semimartingales $(X,\mathbb{F})$ up to an $\mathbb{F}$-stopping time $S$ is obvious. For $(X,\mathbb{F})$ there exists the right (resp., left, symmetric) local time $L^X$ up to $S$ which is a function on $[0,S) \times \mathbb{R}$ into $[0,+\infty)$ such that for every real function $f$ which is the difference of convex functions the \emph{generalized It\^o formula} holds:
\begin{equation}\label{eqn:gen_ito_formula}
	f(X_t) = f(X_0) + \int_0^t f^\prime (X_s) \, \myd X_s + \frac{1}{2} \int_0^t L^X(t,y) \, \myd f^\prime(y)\,, \qquad t < S, \ \mathbf{P}\text{-a.s.} 
\end{equation}
where $f^\prime$ denotes the left (resp., right, symmetric) derivative of $f$. To indicate explicitly which local time we consider, we write $L_+^X$ (resp., $L_-^X$, $\hat{L}^X$) for the right (resp., left, symmetric) local time of $X$. If a formula or statement holds for every type of local time we just use the symbol $L^X$. Note that we can choose $L_+^X$ (resp., $L_-^X$) to be increasing and continuous in $t < S$ and right (resp. left) continuous with left (resp. right) hand limits in $x$. Moreover, we have the relation $\hat{L}^X = (L_+^X + L_-^X)/2$. \\
\indent The local times fulfil
\begin{equation}\label{eqn:int_wrt_loc_time}
	\int_0^t \ind_{\{y\}}(X_s) \, L^X(\myd s,y) = L^X(t,y)\,, \qquad t < S, \ y \in \mathbb{R}, \ \mathbf{P}\text{-a.s.},
\end{equation}
\begin{equation}\label{eqn:loc_time_and_variation_process}
	L_+^X(t,y) - L_-^X(t,y) = 2 \int_0^t \ind_{\{y\}}(X_s) \, \myd V_s\,, \qquad t < S, \, y \in \mathbb{R}, \ \mathbf{P}\text{-a.s.},
\end{equation}
\begin{equation}\label{eqn:loc_time_zero_outside_compact_interval}
	L^X(t,y) = 0, \qquad t < S, \, y \notin \left[\min_{0 \leq s \leq t} X_s, \max_{0 \leq s \leq t} X_s\right], \ \mathbf{P}\text{-a.s.}
\end{equation}
and
\begin{equation}\label{eqn:loc_time_as_limit}
	L^X(t,y) = \lim_{\varepsilon \downarrow 0}\frac{1}{\varepsilon} \int_0^t I^y_\varepsilon(X_s) \, \myd \assPro{X}_s, 
	\qquad t < S, \ y \in \mathbb{R}, \ \mathbf{P}\text{-a.s.},
\end{equation}
where
\[
	I^y_\varepsilon (x) = \ind_{[y,y+\varepsilon)}(x) \quad (\text{resp.,} \ \ind_{(y-\varepsilon, y]}(x), \ 
	\frac{1}{2} \, \ind_{(y-\varepsilon, y+\varepsilon)}(x)), \qquad x \in \mathbb{R}.
\]

In general, SDEs of type (\ref{eqn:SDE_mvd}) admit exploding solutions. Therefore, the convenient state space is the extended real line $\overline{\mathbb{R}} = \mathbb{R}\cup \{-\infty, +\infty\}$ equipped with the $\sigma$-algebra $\mathscr{B}(\overline{\mathbb{R}})$ of Borel subsets. We fix the notion of a solution to Eq. (\ref{eqn:SDE_mvd}) in the following
\begin{defi}\label{def:solution}
	A continuous $(\overline{\mathbb{R}}, \mathscr{B}(\overline{\mathbb{R}}))$-valued stochastic process $(X,\mathbb{F})$ 
	defined on a probability space $(\Omega,\mathcal{F},\mathbf{P})$ is called a solution to Eq. (\ref{eqn:SDE_mvd}) if the following 
	conditions are fulfilled:\medskip
	
	(i) $X_0$ is real-valued. \medskip
	
	(ii)  $X_t = X_{t \wedge S_\infty^X}$, $t \geq 0$,	$\mathbf{P}$-a.s., where $S_\infty^X := \inf\{t \geq 0:|X_t| = +\infty\}$. \medskip
	
	(iii) $(X,\mathbb{F})$ is a semimartingale up to $S_\infty^X$.\medskip
	
	(v)	 There exists a Wiener process $(B,\mathbb{F})$ such that Eq. (\ref{eqn:SDE_mvd}) is satisfied for all $t < S_\infty^X$ $\mathbf{P}$-a.s. 
\end{defi}
\section{The Results}
We recall that in Eq. (\ref{eqn:SDE_mvd}) $L^X$ stands either for the right, the left or the symmetric local time and we treat these three cases simultaneously in the sequel. For the sake of brevity, in the following we write for example $\{\nu \geq 1/2\}$ instead of $\{y \in \mathbb{R}: \nu(\{y\}) \geq 1/2\}$.
\begin{lemma}\label{lemma:loc_time_zero}
	Let $(X,\mathbb{F})$ be a solution of Eq. (\ref{eqn:SDE_mvd}) with right (resp., left, symmetric) local time. Then it holds
	\[
		\begin{split}
													 L_-^X(t,x) = 0, \qquad & t < S_\infty^X, \ x \in \{\nu \geq 1/2\}, \ \mathbf{P}\text{-a.s.} \phantom{\Bigr)}\\
		  \Bigl(\text{resp., } L_-^X(t,x) = 0, \qquad & t < S_\infty^X, \ x \in \{\nu < -1/2\},  \ \mathbf{P}\text{-a.s.}, \\
		  										 L_-^X(t,x) = 0, \qquad & t < S_\infty^X, \ x \in\{|\nu| > 1 \text{ or } \nu = 1\}, \ \mathbf{P}\text{-a.s.}\Bigr)
		\end{split}
	\]
	and
	\[
		\begin{split}
													 L_+^X(t,x) = 0, \qquad & t < S_\infty^X, \ x \in \{\nu > 1/2\}, \ \mathbf{P}\text{-a.s.} \phantom{\Bigr)}\\
			\Bigl(\text{resp., } L_+^X(t,x) = 0, \qquad & t < S_\infty^X, \ x \in \{\nu \leq -1/2\},  \ \mathbf{P}\text{-a.s.}, \\
													 L_+^X(t,x) = 0, \qquad & t < S_\infty^X, \ x \in \{|\nu| > 1 \text{ or } \nu = -1\}, \ \mathbf{P}\text{-a.s.} \Bigr)
		\end{split}
	\]
\end{lemma}
\begin{proof}
Using (\ref{eqn:int_wrt_loc_time}) and (\ref{eqn:loc_time_and_variation_process}), we see
\[\begin{split}
	L_+^X(t,x)-L_-^X(t,x) &= 2 \int_0^t \ind_{\{x\}} (X_s) \int_\mathbb{R} L^X(\myd s, y) \, \nu(\myd y) \\
												&= 2 \, L^X(t, x) \, \nu(\{x\}), \qquad t < S_\infty^X, \ x \in \mathbb{R}, \ \mathbf{P}\text{-a.s.}
\end{split}\]
and it follows
\[
	\begin{split}
		\bigl(1-2\nu(\{x\})\bigr) \, L_+^X(t,x) &= L_-^X(t,x), \qquad t < S_\infty^X, \ x \in \mathbb{R}, \ \mathbf{P}\text{-a.s.} \phantom{\Bigr)}\\
		\Bigl(\text{resp., }
		\bigl(1+2\nu(\{x\})\bigr) \, L_-^X(t,x) &= L_+^X(t,x), \qquad t < S_\infty^X, \ x \in \mathbb{R}, \ \mathbf{P}\text{-a.s.}, \\
		 \bigl(1-\nu(\{x\})\bigr) \, L_+^X(t,x)  &= \bigl(1+\nu(\{x\})\bigr)\,L_-^X(t,x), \qquad t < S_\infty^X, \ x \in \mathbb{R}, \ \mathbf{P}\text{-a.s.} \Bigr)\\
	\end{split}
\]
which, together with the non-negativity of the local times, implies the claims.
\end{proof}
The following theorem is the main result of the present note.
\begin{theorem}\label{theorem:reflection_and_absorbing}
	Let $(X,\mathbb{F})$ be a solution of Eq. (\ref{eqn:SDE_mvd}) with right (resp., left, symmetric) local time started at $x_0 \in \mathbb{R}$. Then 
	the following statements are satisfied: \medskip
	
	(i)  Assume $\nu(\{x_0\}) \geq 1/2$ (resp., $\nu(\{x_0\}) \leq -1/2$, $|\nu(\{x_0\})| \geq 1$). Then it holds
			 \[
			 	 \begin{split}
													  	&X_t \geq x_0, \ t \geq 0, \ \mathbf{P}\text{-a.s.} \phantom{\Bigr)}\\
		  	 \Bigl(\text{resp., } &X_t \leq x_0,  \ t \geq 0, \ \mathbf{P}\text{-a.s.}, \\
		  										  	&X_t \leq x_0,  \ t \geq 0, \ \mathbf{P}\text{-a.s.} \text{ if } \nu(\{x_0\}) \leq -1 \text{ and }
		  										  	 X_t \geq x_0,  \ t \geq 0, \ \mathbf{P}\text{-a.s.} \text{ if } \nu(\{x_0\}) \geq 1 \Bigr)
				 \end{split}
			 \]
			 i.e., the point $x_0$ is reflecting. \medskip
			 
	(ii) Assume $\nu(\{x_0\}) > 1/2$ (resp., $\nu(\{x_0\}) < -1/2$, $|\nu(\{x_0\})| > 1$). Then it holds 
			 \[
			 	X_t = x_0, \qquad t \geq 0, \ \mathbf{P}\text{-a.s.}, 
			 \]
			 i.e., the point $x_0$ is absorbing. In particular, $b(x_0) = 0$ must be fulfilled.
\end{theorem}
\begin{proof}
\indent \textbf{1)} Let $(X,\mathbb{F})$ be a solution of Eq. (\ref{eqn:SDE_mvd}) started at $x_0 \in \mathbb{R}$. At first we reduce the problem to the case $x_0 = 0$. For the process $(X - x_0, \mathbb{F})$ it clearly holds
\[
	\begin{split}
		X_t -x_0 &= \int_0^t b_{x_0} (X_s - x_0) \, \myd B_s 
									 + \int_\mathbb{R} L^{X-x_0}(t,y-x_0) \, \nu(\myd y), \qquad t < S_\infty^X, \ \mathbf{P}\text{-a.s.},
	\end{split}
\]
where $b_{x_0}(x) := b(x+x_0)$, $x \in \mathbb{R}$, and we have used the relation $L^X(t,x) = L^{X-x_0}(t,x-x_0)$, $t<S_\infty^X$, $x \in \mathbb{R}$, $\mathbf{P}$-a.s. which can be easily deduced for example by exploiting (\ref{eqn:loc_time_as_limit}). Introducing the drift measure $\nu_{x_0}$ via $\nu_{x_0}(B) := \nu(B+x_0)$,\footnote{For $B \subseteq \mathbb{R}$, we set $B+x_0 := \{x+x_0 :\; x \in B\}$.} $B\in\mathscr{B}([-N,N])$, $N \in \mathbb{N}$, we obtain 
\[
	\begin{split}
		X_t -x_0 &= \int_0^t b_{x_0} (X_s - x_0) \, \myd B_s 
									 + \int_\mathbb{R} L^{X-x_0}(t,y) \, \nu_{x_0}(\myd y), \qquad t < S_\infty^X, \ \mathbf{P}\text{-a.s.}
	\end{split}
\]
Hence, $(X-x_0,\mathbb{F})$ is also a solution to an equation of type (\ref{eqn:SDE_mvd}) but started at zero and the drift measure satisfies $\nu_{x_0}(\{0\}) = \nu(\{x_0\})$. Therefore, without loss of generality we assume $x_0 = 0$ in the following parts of the proof. \\
\indent \textbf{2)} Now let us consider the case of the right local time in Eq. (\ref{eqn:SDE_mvd}). To prove (i) we assume $\nu(\{0\}) \geq 1/2$, we set $Z_t = X_t \wedge 0$, $t \geq 0$. We apply Tanaka's formula (see (\ref{eqn:gen_ito_formula}) for $f(x) = -x^- = x \wedge 0$, $x \in \mathbb{R}$) for the left local time to obtain
\[
	Z_t = \int_0^t \ind_{(-\infty,0)}(X_s) \, b(X_s) \, \myd B_s 
									 + \int_0^t \ind_{(-\infty,0)}(X_s) \int_\mathbb{R} L_+^X(\myd s,y) \, \nu(\myd y) - \frac{1}{2} L_-^X(t,0), 
\]
$t < S_\infty^X, \ \mathbf{P}\text{-a.s.}$ Lemma \ref{lemma:loc_time_zero} shows that the left local time of $X$ in zero vanishes and we can write
\[\begin{split}
	Z_t &= \int_0^t \ind_{(-\infty,0)}(Z_s) \, b(Z_s) \, \myd B_s 
										+ \int_\mathbb{R} \int_0^t \ind_{(-\infty,0)}(X_s) \, L_+^X(\myd s,y) \, \nu(\myd y) \\
								 &= \int_0^t \ind_{(-\infty,0)}(Z_s) \, b(Z_s) \, \myd B_s 
										+ \int_\mathbb{R} L_+^Z(t,y)\, \ind_{(-\infty,0)}(y)  \, \nu(\myd y), \qquad t < S_\infty^X, \ \mathbf{P}\text{-a.s.},
\end{split}\]
where we used (\ref{eqn:int_wrt_loc_time}) and the easy to check relation $L_+^X(t,y) = L_+^Z(t,y)$, $t < S_\infty^X$, $y < 0$, $\mathbf{P}$-a.s. (use e.g. (\ref{eqn:loc_time_as_limit})). Introducing the sets
\[
	A_1 := \{y \in (-\infty,0) : \nu(\{y\}) \geq 1/2\}, \quad A_2 := \{y \in (-\infty,0) : \nu(\{y\}) = 1/2\}
\]
and the set function $\mu(\myd y) := \ind_{(-\infty,0) \setminus A_1} (y) \, \nu (\myd y)$, in the decomposition of $Z$ we can split the last integral and, using Lemma \ref{lemma:loc_time_zero}, we can write
\[
	Z_t = \int_0^t \ind_{(-\infty,0)}(Z_s) \, b(Z_s) \, \myd B_s 
										+ \int_\mathbb{R} L_+^{Z}(t,y) \, \mu(\myd y)
										+ \int_\mathbb{R} L_+^{Z}(t,y) \, \ind_{A_2}(y)  \, \nu(\myd y),
\]
$t < S_\infty^X$, $\mathbf{P}$-a.s. Noting that $\mu(\{x\}) < 1/2$, $x \in \mathbb{R}$, we define the strictly positive function $g$ as the unique solution of the integral equation (\ref{eqn:integral_eqn}) with respect to $\mu$ and denote by $G(x) = \int_0^x g(y)\, \myd y$, $x \in \overline{\mathbb{R}}$, its strictly increasing and continuous primitive. Then, with the notation 
\[
	M_t := \int_0^t \ind_{(-\infty,0)}(Z_s) \, b(Z_s) \, \myd B_s, \qquad t < S_\infty^X\,,
\]
and noting that $G$ restricted to $\mathbb{R}$ is the difference of convex functions, due to the generalized It\^o formula for the right local time and since $\myd g (y) = -2\, g(y-) \, \mu(\myd y)$, for $Y = G(Z)$ it holds
\begin{equation}\label{eqn:spacetrans_good_points}
 \begin{split}
	Y_t	&= \int_0^t g(Z_s) \, \myd M_s 
							 + \int_0^t g(Z_s-) \int_\mathbb{R} L_+^{Z}(\myd s,y) \, \mu(\myd y) \\
						&\phantom{=======}+ \int_0^t g(Z_s-) \int_\mathbb{R} L_+^{Z}(\myd s,y) \, \ind_{A_2}(y)  \, \nu(\myd y)
							 - \int_\mathbb{R} L_+^{Z}(t,y) \, g(y-) \, \mu(\myd y)\\
						&= \int_0^t g(Z_s) \, \myd M_s  
								+  \int_\mathbb{R} L_+^{Z}(t,y) \, g(y-) \, \ind_{A_2}(y)  \, \nu(\myd y),
 \end{split}
\end{equation}
$t < S_\infty^X$, $\mathbf{P}$-a.s. Clearly, since we have $Z_t \leq 0$, $t \geq 0$, and since $G$ maps $(-\infty,0]$ into $(-\infty,0]$, it follows $Y_t \leq 0$, $t \geq 0$. Using that $g$ is strictly positive, we conclude
\begin{equation}\label{eqn:aux_2}
	\int_0^t g(Z_s) \, \myd M_s	\leq Y_t \leq 0, \qquad t < S_\infty^X\, \ \mathbf{P}\text{-a.s.}
\end{equation}
The process $\int_0^\cdot g(Z_s) \, \myd M_s$ being a non-positive continuous local martingale up to $S_\infty^X$ starting at zero must be zero $\mathbf{P}$-a.s. which implies $\int_0^t g^2(Z_s)\, \myd \assPro{M}_s = 0$, $t < S_\infty^X$, $\mathbf{P}$-a.s., and thus $M_t = 0$, $t < S_\infty^X$, $\mathbf{P}$-a.s. Hence, $Z$ is a process of locally bounded variation on $[0,S_\infty^X)$ and we obtain 
\[
	L_\pm^{Z}(t,y) = 0, \qquad t < S_\infty^X, \ y \in \mathbb{R}, \ \mathbf{P}\text{-a.s.}
\]
Therefore, it holds $Z_t = 0$, $t < S_\infty^X$, $\mathbf{P}$-a.s., which means $X_t \geq 0$, $t \geq 0$, $\mathbf{P}$-a.s. and (i) is proven. \\
\indent \textbf{3)} Now, still considering the case of the right local time in Eq. (\ref{eqn:SDE_mvd}), we show (ii) which is why we assume $\nu(\{0\}) > 1/2$. From (i) we derive $X_t \geq 0$, $t \geq 0$, $\mathbf{P}$-a.s., and hence via (\ref{eqn:loc_time_zero_outside_compact_interval})
\[
	X_t	= \int_0^t b(X_s) \, \myd B_s + \int_\mathbb{R} \ind_{[0,+\infty)}(y) \, L_+^X(t,y) \, \nu(\myd y), 
	\qquad t < S_\infty^X, \ \mathbf{P}\text{-a.s.}
\]
Setting
\[
	A_1 := \{y \in [0,+\infty) : \nu(\{y\}) \geq 1/2\}, \quad A_2 := \{y \in [0,+\infty) : \nu(\{y\}) = 1/2\}
\]
and defining the set function $\mu(\myd y) := \ind_{[0,+\infty) \setminus A_1} (y) \, \nu (\myd y)$, we can write
\[
	X_t	= \int_0^t  b(X_s) \, \myd B_s + \int_\mathbb{R} L_+^X(t,y) \, \mu(\myd y) 
					+ \int_\mathbb{R} \ind_{A_2}(y) \, L_+^X(t,y) \, \nu(\myd y), 
\]
$t < S_\infty^X, \ \mathbf{P}\text{-a.s.}$, since the right local time $L_+^X$ of $X$ vanishes on $A_1 \setminus A_2$ by Lemma \ref{lemma:loc_time_zero}. Note that $\mu$ satisfies (\ref{eqn:cond_atoms}). Let $g$ be the unique solution of (\ref{eqn:integral_eqn}) with respect to $\mu$ and $G(x) := \int_0^x g(y)\, \myd y$, $x \in \overline{\mathbb{R}}$, its strictly increasing and continuous primitive. Using similar arguments as in (\ref{eqn:spacetrans_good_points}), for $Y := G(X)$ we obtain
\[
	Y_t = \int_0^t g(X_s) \, b(X_s) \, \myd B_s 
							+ \int_\mathbb{R} g(y-) \, L_+^X(t,y) \, \ind_{A_2}(y) \, \nu(\myd y), \qquad t < S_\infty^X, \ \mathbf{P}\text{-a.s.}
\]
Having in mind that $0 \notin A_2$ and that $\nu$ is a finite signed measure on $\mathscr{B}([-1,1])$, we conclude\footnote{$\inf \emptyset := + \infty$.} $c := \inf A_2 > 0$. Therefore, $\tau_c := \inf\{t \geq 0 : X_t = c\}$ is a strictly positive $\mathbb{F}$-stopping time. For the stopped process $Y_t^{\tau_c} := Y_{\tau_c \wedge t}$, $t \geq 0$, due to (\ref{eqn:loc_time_zero_outside_compact_interval}) it holds
\[
	Y_t^{\tau_c} = \int_0^{\tau_c \wedge t} g(X_s) \, b(X_s) \, \myd B_s, 
											 \qquad t < S_\infty^X, \ \mathbf{P}\text{-a.s.}
\]
Since $G$ maps $[0,+\infty)$ into $[0,+\infty)$, the last relation means that $Y^{\tau_c}$ is a non-negative continuous local martingale up to $S_\infty^X$ started at zero, which implies immediately $Y_t^{\tau_c} = 0$, $t \geq 0$, $\mathbf{P}$-a.s. Hence, because of $\tau_c = \inf\{t \geq 0: Y_t = G(c)\}$ and $G(c) > 0$, we conclude $\tau_c = +\infty$ $\mathbf{P}$-a.s. Finally, we obtain $X_t = 0$, $t \geq 0$, $\mathbf{P}$-a.s. Moreover, we have
\[
	0 = \assPro{X}_t = \int_0^t b^2(X_s) \, \myd s = b^2(0) \, t, \qquad t \geq 0, \ \mathbf{P}\text{-a.s.},
\]
which is only possible if $b(0)=0$. \\
\indent \textbf{4)} Now we prove (i) and (ii) for the case of the left local time in Eq. (\ref{eqn:SDE_mvd}). We recall that we only need to consider the initial value $x_0 = 0$. So, we treat the equation
\[
	X_t = \int_0^t b(X_s) \, \myd B_s + \int_\mathbb{R} L_-^X(t,y) \, \nu(\myd y),
\]
where we assume $\nu(\{0\}) \leq -1/2$. Let $(X,\mathbb{F})$ be a solution of this equation. Introducing $\widetilde{b}(x) = b(-x)$, $x \in \mathbb{R}$, and the Wiener process $\widetilde{B} = -B$, for $-X$ it holds
\[
	-X_t = \int_0^t \widetilde{b}(-X_s) \, \myd \widetilde{B}_s - \int_\mathbb{R} L_-^X(t,y)\,\nu(\myd y), 
	\qquad t < S_\infty^X, \ \mathbf{P}\text{-a.s.}
\]
We define $\widetilde{\nu}(A) := -\nu(-A)$,\footnote{For $A \subseteq \mathbb{R}$, we set $-A := \{-x : x \in A\}$.} $A \in \mathscr{B}([-N,N])$, $N \in \mathbb{N}$. Then, since $S_\infty^X = S_\infty^{-X}$ and since $L_-^X(t,y) = L_+^{-X}(t,-y)$, $t < S_\infty^X$, $y \in \mathbb{R}$, $\mathbf{P}$-a.s., which follows immediately from (\ref{eqn:loc_time_as_limit}), we get
\[
	-X_t = \int_0^t \widetilde{b}(-X_s) \, \myd \widetilde{B}_s + \int_\mathbb{R} L_+^{-X}(t,y)\,\widetilde{\nu}(\myd y), 
	\qquad t < S_\infty^{-X}, \ \mathbf{P}\text{-a.s.}
\]
Hence, $(-X, \mathbb{F})$ is a solution of an equation of type (\ref{eqn:SDE_mvd}) with right local time started at zero and for the drift measure $\widetilde{\nu}$ it holds $\widetilde{\nu}(\{0\}) = - \nu(\{0\}) \geq 1/2$. Using step 2) and 3) of the proof yields $X_t \leq 0$, $t < S_\infty^X$, $\mathbf{P}$-a.s. if $\nu(\{0\}) \leq 1/2$ and $X_t = 0$, $t < S_\infty^X$, $\mathbf{P}$-a.s. as well as $b(0)=0$ if $\nu(\{0\}) < -1/2$.\\
\indent \textbf{5)} For the proof in case of the symmetric local time in Eq. (\ref{eqn:SDE_mvd}) all tools are already presented above. But one must be careful in adapting them. We provide the idea but the details are left to the reader. Again without loss of generality we assume $x_0 = 0$. Let $(X,\mathbb{F})$ be a solution to
\[
	X_t = \int_0^t b(X_s) \, \myd B_s + \int_\mathbb{R} \hat{L}^X(t,y) \, \nu(\myd y).
\]
We first assume that $|\nu(\{0\})| > 1$ or $\nu(\{0\}) = 1$. Similar as above in step 2) for $Z_t = X_t \wedge 0$, $t \geq 0$, we deduce
\[
	Z_t = \int_0^t \ind_{(-\infty,0)}(Z_s) \, b(Z_s) \, \myd B_s 
				+ \int_\mathbb{R} \hat{L}^Z(t,y) \, \ind_{(-\infty,0)}(y)  \, \nu(\myd y), \qquad t < S_\infty^X, \ \mathbf{P}\text{-a.s.},
\]
Furthermore, with the help of the sets
\begin{gather*}
	A_1 := \{y \in (-\infty,0) : |\nu(\{y\})| \geq 1\}, \quad  
	A_2 := \{y \in (-\infty,0) : \nu(\{y\}) = 1\}, \\
	A_3 := \{y \in (-\infty,0) : \nu(\{y\}) = -1\},
\end{gather*}
the set function $\mu(\myd y) := \ind_{(-\infty,0) \setminus A_1} (y) \, \nu (\myd y)$ and Lemma \ref{lemma:loc_time_zero}, we can write 
\[
	Z_t = \int_0^t \ind_{(-\infty,0)}(Z_s) \, b(Z_s) \, \myd B_s 
										+ \int_\mathbb{R} \hat{L}^{Z}(t,y) \, \mu(\myd y)
										+ \int_{A_2} L_+^{Z}(t,y) \, \nu(\myd y)
										+ \int_{A_3} L_-^{Z}(t,y) \, \nu(\myd y),
\]
$t < S_\infty^X$, $\mathbf{P}$-a.s. Using the unique solution $g$ of the integral equation (\ref{eqn:integral_eqn}) in case of the symmetric local time with respect to $\mu$, its primitive $G(x) := \int_0^x g(y)\,\myd y$, $x \in \overline{\mathbb{R}}$, and setting $Y := G(X)$, analogously as in (\ref{eqn:spacetrans_good_points}) we get
\[
 Y_t	= \int_0^t g(Z_s) \, \myd M_s + \int_{A_2} L_+^{Z}(t,y) \, g(y)  \, \nu(\myd y) + \int_{A_3} L_-^{Z}(t,y) \, g(y)  \, \nu(\myd y), 
 \qquad t < S_\infty^X, \ \mathbf{P}\text{-a.s.}
\]
where $M := \int_0^\cdot \ind_{(-\infty,0)}(Z_s) \, b(Z_s) \, \myd B_s$. Note that, because of $\mu(\{y\}) = 0$, $y \in A_2 \cup A_3$, the function $g$ is continuous in the points of $A_2 \cup A_3$. Since $\nu$ is a finite signed measure on $\mathscr{B}([-N,N])$, $N\in \mathbb{N}$, we deduce\footnote{$\sup \emptyset := -\infty$.} $c:= \sup A_3 <0$ and hence the $\mathbb{F}$-stopping time 
\[
	\tau_c := \inf\{t\geq 0 : Z_t = c\} = \inf\{t\geq 0 : Y_t = G(c)\}
\]
is strictly positive. Clearly, due to (\ref{eqn:loc_time_zero_outside_compact_interval}) we have
\[
 Y_t	= \int_0^t g(Z_s) \, \myd M_s + \int_{A_2} L_+^{Z}(t,y) \, g(y)  \, \nu(\myd y) \qquad t < \tau_c \wedge S_\infty^X, \ \mathbf{P}\text{-a.s.}
\]
and similarly to (\ref{eqn:aux_2}) and the lines thereafter we conclude $Z_t = 0$, $t < \tau_c \wedge S_\infty^X$, $\mathbf{P}$-a.s. By the definition of $\tau_c$ this gives $Z_t = 0$, $t < S_\infty^X$, $\mathbf{P}$-a.s. and therefore $X_t \geq 0$, $t \geq 0$, $\mathbf{P}$-a.s. \\
\indent Now we assume $|\nu(\{0\})| > 1$ or $\nu(\{0\}) = -1$. As shown in step 4), $-X$ is a solution of Eq. (\ref{eqn:SDE_mvd}) with drift measure $\widetilde{\nu}$ defined by $\widetilde{\nu}(A) = - \nu(-A)$, $A \in \mathscr{B}([-N,N])$, $N \in \mathbb{N}$. Hence, from the result just proven we obtain $-X_t \geq 0$, $\geq 0$, $\mathbf{P}$-a.s. Summarizing, this yields
\[
	X_t = 0, \ t \geq 0, \ \mathbf{P}\text{-a.s. and necessarily $b(0)=0$ if $|\nu(\{0\})| > 1$,}
\]
\[
	X_t \geq 0, \ t \geq 0, \ \mathbf{P}\text{-a.s. if $\nu(\{0\}) = 1$,}
\]
and
\[
	X_t \leq 0, \ t \geq 0, \ \mathbf{P}\text{-a.s. if $\nu(\{0\}) = -1$.}
\]
This finishes the proof of (i) and (ii) for symmetric local time.
\end{proof}

Introducing the $\mathbb{F}$-stopping time $\tau_x := \inf\{t \geq 0 : X_t = x\}$ for a solution $(X,\mathbb{F})$ of Eq. (\ref{eqn:SDE_mvd}), we can state more generally the following corollary to Theorem \ref{theorem:reflection_and_absorbing}.
\begin{corollary}\label{cor:reflection_and_absorbing}
	Let $(X,\mathbb{F})$ be a solution of Eq. (\ref{eqn:SDE_mvd}) with right (resp., left, symmetric) local time and arbitrary initial condition $X_0$. 
	Then the following statements are satisfied: \medskip
	
	(i)  If $x \in \mathbb{R}$ is such that $\nu(\{x\}) \geq 1/2$ (resp., $\nu(\{x\}) \leq -1/2$, $|\nu(\{x\})| \geq 1$), then on $\{\tau_x < +\infty\}$  			it holds
			 \[
			 	 \begin{split}
													  	&X_{\tau_x + t} \geq x, \ t \geq 0, \ \mathbf{P}\text{-a.s}\phantom{\Bigr)} \\
		  	 \Bigl(\text{resp., } &X_{\tau_x + t} \leq x,  \ t \geq 0, \ \mathbf{P}\text{-a.s},\\
		  										  	&X_{\tau_x + t} \leq x,  \ t \geq 0,\ \mathbf{P}\text{-a.s} \text{ if } \nu(\{x\}) \leq -1
		  										  	\text{ and } X_{\tau_x + t} \geq x, \ t \geq 0, \ \mathbf{P}\text{-a.s} \text{ if } \nu(\{x\}) \geq 1 \Bigr).
				 \end{split}
			 \]

	(ii) If $x \in \mathbb{R}$ is such that $\nu(\{x\}) > 1/2$ (resp., $\nu(\{x\}) < -1/2$, $|\nu(\{x\})| > 1$), then it holds 
			 \[
			 		X_{\tau_x + t} = x, \qquad t \geq 0, \ \mathbf{P}\text{-a.s. on } \{\tau_x < +\infty\}.	
			 \]
			 In particular, if $\mathbf{P}(\{\tau_x < +\infty\})>0$, then $b(x) = 0$ must be fulfilled.
\end{corollary}
\begin{proof} 
We only give the idea of the proof. The details are left to the reader. If $\mathbf{P}(\{\tau_x < +\infty\}) > 0$ is satisfied, then we use the trace of the underlying probability space with respect to $\widetilde{\Omega} = \{\tau_x < +\infty\}$. The process $X_{\tau_x + t}$, $t\geq 0$, considered on $\widetilde{\Omega}$, which is adapted to $\widetilde{\mathbb{F}} = (\mathcal{F}_t \cap \widetilde{\Omega})_{t \geq 0}$, is again a solution of Eq. (\ref{eqn:SDE_mvd}) but started at $x$. Hence, we can apply Theorem \ref{theorem:reflection_and_absorbing} to obtain the statements of the corollary.
\end{proof}
Concerning the question of non-existence of solutions of Eq. (\ref{eqn:SDE_mvd}), we see from Theorem \ref{theorem:reflection_and_absorbing} and its Corollary \ref{cor:reflection_and_absorbing} that the existence of a point $x \in \mathbb{R}$ with $\nu(\{x\}) > 1/2$ (resp., $\nu(\{x\}) < -1/2$, $|\nu(\{x\})| > 1$), in general, does not imply automatically that there is no solution. In general, there can be solutions which do not reach $x$ or, if $b(x)=0$ is satisfied, which stay in $x$ after reaching this level. But it holds the following
\begin{corollary}\label{cor:non_existence_of_a_solution}
	If $x \in \mathbb{R}$ is such that $\nu(\{x\}) > 1/2$ (resp., $\nu(\{x\}) < -1/2$, $|\nu(\{x\})| > 1$) and $b(x) \neq 0$, then 	
	there is no solution $(X,\mathbb{F})$ of Eq. (\ref{eqn:SDE_mvd}) with right (resp., left, symmetric) local time and arbitrary initial condition 
	$X_0$ which satisfies $\mathbf{P}(\{\tau_x < +\infty\}) > 0$. In particular, there is no solution started at $X_0 = x$.
\end{corollary}
\begin{remark}
	(i) In the special case of a drift measure $\nu = \beta \, \delta_0$ where $|\beta|=1$ and $\delta_0$ denotes the Dirac measure in zero and
	the symmetric local time in Eq. 
	(\ref{eqn:SDE_mvd}) the result of Corollary \ref{cor:reflection_and_absorbing}(i) was already presented in \cite{blei_bessel_2011}, Lemma 2.24. 
	Moreover, Corollary \ref{cor:non_existence_of_a_solution} contains \cite{blei_bessel_2011}, Lemma 2.25, which deals with the non-existence of a 
	solution to Eq. (\ref{eqn:SDE_mvd}) with symmetric local time and drift measure $\nu = \beta \, \delta_0$ such that $|\beta|>1$. 
	
	(ii) In W. Schmidt \cite{schmidt:1989} one-dimensional stochastic differential equations with reflecting barriers, more precisely, equations of the 
	form
	\begin{align}\label{eqn:SDE_refl_schmidt}
		\left\{
			\begin{aligned}
				\text{(i)} 		&\phantom{=}\displaystyle X_t = X_0 + \int_0^t b(X_s) \, \myd B_s + L_t - R_t, \\
				\text{(ii)} 	&\phantom{=}X_t \in [r_1,r_2], \\
				\text{(iii)} &\phantom{=}\text{$L_t$, $R_t$ are increasing process with $L_0 = R_0 = 0$ and} \\
											&\phantom{===}\begin{gathered}
											 	\int_0^t \ind_{\{r_1\}} (X_s) \, \myd L_s = L_t, \quad \int_0^t \ind_{\{r_2\}} (X_s) \, \myd R_s = R_t, \quad t \geq 0,
											 \end{gathered}
			\end{aligned}		
		\right.
	\end{align}
	where $r_1 < r_2$, were studied. Clearly, $(X,\mathbb{F})$ on $(\Omega,\mathcal{F},\mathbf{P})$ is called a solution of Eq. 
	(\ref{eqn:SDE_refl_schmidt}) if there exists a Wiener 
	process $(B,\mathbb{F})$ and two processes $(L,\mathbb{F})$ and $(R,\mathbb{F})$ such that (\ref{eqn:SDE_refl_schmidt}) is satisfied. For a solution 
	$(X,\mathbb{F})$ of Eq. (\ref{eqn:SDE_refl_schmidt}) it is not difficult to check that $L$ and $R$ are just the symmetric local times of $X$ in 
	$r_1$ and $r_2$, respectively. Our results, especially Corollary \ref{cor:reflection_and_absorbing}(i), now show that Eq. 
	(\ref{eqn:SDE_refl_schmidt}) even coincides with Eq. (\ref{eqn:SDE_mvd}) for the diffusion coefficient $b$, symmetric local time, drift measure 
	$\nu = \delta_{r_1} - \delta_{r_2}$ and initial condition $X_0 \in [r_1,r_2]$, i.e., with the equation
	\begin{equation}\label{eqn:sde_refl}
		X_t = X_0 + \int_0^t b(X_s) \, \myd B_s + \hat{L}^X(t,r_1) - \hat{L}^X(t,r_2), \qquad X_0 \in [r_1,r_2]\,.
	\end{equation}
	For Eq. (\ref{eqn:sde_refl}) the condition (\ref{eqn:SDE_refl_schmidt})(ii) needs not be specified since it is satisfied for any solution which 
	follows via Corollary \ref{cor:reflection_and_absorbing}(i). Moreover, (\ref{eqn:SDE_refl_schmidt})(iii) holds because of 
	(\ref{eqn:int_wrt_loc_time}).
	
	(iii) Similar as in the preceding remark the non-negativity condition $X_t \geq 0$, $t \geq 0$, of the solution to the equation
	\[
		X_t = X_0 + \int_0^t \ind_{(0,+\infty)}(X_s) \, \myd B_s + a \, \int_0^t \ind_{\{0\}} (X_s) \, \myd s,
	\]
	with $X_0 \geq 0$ and $a \geq 0$, which is studied in R. Chitashvili \cite{chitashvili}, can be dropped. More detailed, for a solution 
	$(X,\mathbb{F})$ of this equation by (\ref{eqn:loc_time_as_limit}) it follows $L_-^X(t,0) = 0$, $t \geq 0$, $\mathbf{P}$-a.s. Combined with 
	(\ref{eqn:loc_time_and_variation_process}) this implies
	\[
		\frac{1}{2}\, L_+^X(t,0) = a \, \int_0^t \ind_{\{0\}} (X_s) \, \myd s, \qquad t \geq 0, \ \mathbf{P}\text{-a.s.}
	\]
	Hence, $(X,\mathbb{F})$ also solves Eq. (\ref{eqn:SDE_mvd}) with diffusion coefficient $b=\ind_{(0,+\infty)}$, right local time and drift measure 
	$\nu = 1/2 \,\delta_0$. Therefore, we can conclude $X_t \geq 0$, $t \geq 0$, $\mathbf{P}$-a.s. if $X_0 \geq 0$.	\mydia
\end{remark}
With the help of Lemma \ref{lemma:loc_time_zero} we can also give a relation between the different versions of Eq. (\ref{eqn:SDE_mvd}). More precisely, we give a relation between Eq. (\ref{eqn:SDE_mvd}) using the right local time and Eq. (\ref{eqn:SDE_mvd}) with symmetric local time. With the following proposition we complement \cite{bass_chen}, Theorem 2.2(a).
\begin{proposition}\label{prop:equiv_mvd_right_and_sym_loc_time}
	(i) $(X,\mathbb{F})$ is a solution of Eq. (\ref{eqn:SDE_mvd}) with right local time and drift measure $\nu$ if and only if it is 
			a solution of Eq. (\ref{eqn:SDE_mvd}) with symmetric local time and drift measure
			\[
		 		\hat{\nu}(\myd y) = \left( \frac{1}{2-2\nu(\{y\})}\, \ind_{\{z\in\mathbb{R} : \, \nu(\{z\}) < 1\}}(y) 
		 														+  \ind_{\{z\in\mathbb{R} : \, \nu(\{z\}) \geq 1\}}(y)\right) 2\nu(\myd y)\,.
			\]

	(ii) $(X,\mathbb{F})$ is a solution of Eq. (\ref{eqn:SDE_mvd}) with symmetric local time and drift measure 
			 $\hat{\nu}$ satisfying $\hat{\nu}(\{x\}) \neq -1$, $x \in \mathbb{R}$, if and only if it is a solution of Eq. (\ref{eqn:SDE_mvd}) 
			 with right local time and drift measure 
			 \[
		 	 	\nu(\myd y) = \left( \frac{2}{1+\hat{\nu}(\{y\})}\, \ind_{\{z\in\mathbb{R} : \, \hat{\nu}(\{z\}) > -1\}}(y) 
		 					- \ind_{\{z\in\mathbb{R} : \, \hat{\nu}(\{z\}) < -1\}}(y)\right) \frac{1}{2}\hat{\nu}(\myd y)\,.
			 \]
\end{proposition}
\begin{proof}
\textbf{1)} Let $(X,\mathbb{F})$ be a solution of Eq. (\ref{eqn:SDE_mvd}) with right local time, then by Lemma \ref{lemma:loc_time_zero} it is
\[
	\hat{L}^X(t,y) = (L^X_+(t,y) + L_-^X(t,y))/2 = 0, \qquad t < S_\infty^X, \ y \in \{\nu> 1/2\}, \ \mathbf{P}\text{-a.s.}
\]
and
\[
	\hat{L}^X(t,y) = \frac{1}{2} \, L^X_+(t,y) = (1-\nu(\{y\})) \, L_+^X(t,y), \qquad t < S_\infty^X, \ y \in \{\nu = 1/2\}, \ \mathbf{P}\text{-a.s.}
\]
Moreover, via (\ref{eqn:loc_time_and_variation_process}) we see
\[
	L_+^X(t,y) - L_-^X(t,y) = 2 \, L_+^X(t,y) \, \nu(\{y\}), \qquad t < S_\infty^X, \ y \in \{\nu < 1/2\}, \ \mathbf{P}\text{-a.s.}
\]
and we can conclude
\[
	\hat{L}^X(t,y) = \left(1-\nu(\{y\})\right) \, L_+^X(t,y), \qquad t < S_\infty^X, \ y \in \{\nu < 1/2\}, \ \mathbf{P}\text{-a.s.}
\]
Summarizing, we obtain that $(X,\mathbb{F})$ fulfils Eq. (\ref{eqn:SDE_mvd}) with symmetric local time and drift measure $\hat{\nu}$ as given under (i). \\
\indent \textbf{2)} If $(X,\mathbb{F})$ is a solution of Eq. (\ref{eqn:SDE_mvd}) with symmetric local time and drift measure $\hat{\nu}$, then Lemma \ref{lemma:loc_time_zero} implies
\[
	L_{\pm}^X(t,y) = \hat{L}^X(t,y) = 0, \qquad t < S_\infty^X, \ y \in \{|\hat{\nu}| > 1\}, \ \mathbf{P}\text{-a.s.},
\]
and
\[
	L_-^X(t,y) = 0, \qquad t < S_\infty^X, \ y \in \{\hat{\nu} = 1\}, \ \mathbf{P}\text{-a.s.}
\]
Furthermore, via (\ref{eqn:loc_time_and_variation_process}) we obtain
\[
	L_+^X(t,y) - L_-^X(t,y) = 2 \, \hat{L}^X(t,y) \, \hat{\nu}(\{y\}), \qquad t < S_\infty^X, \ y \in \mathbb{R}, \ \mathbf{P}\text{-a.s.},
\]
and hence
\[
	L_+^X(t,y) = (1 + \hat{\nu}(\{y\})) \, \hat{L}(t,y), \qquad t < S_\infty^X, \ y \in \{|\hat{\nu}| < 1\}, \ \mathbf{P}\text{-a.s.}
\]
These observations mean, if we additionally assume $\hat{\nu}(\{x\}) \neq -1$, $x \in \mathbb{R}$, then $(X,\mathbb{F})$ satisfies Eq. (\ref{eqn:SDE_mvd}) with right local time and drift measure $\nu$ as given in (ii). \\
\indent \textbf{3)} Note that $\hat{\nu}$ as defined in (i) satisfies $\hat{\nu}(\{x\}) > -1$, $x \in \mathbb{R}$, and it holds $\hat{\nu}(\{x\}) > 1$ if and only if $\nu(\{x\}) > 1/2$. Hence, the remaining part of assertion (i) follows by an application of step 2) of the proof. To prove the sufficiency in (ii), we remark that $\nu$ introduced in (ii) satisfies $\nu(\{x\}) > 1/2$ if and only if $|\hat{\nu}(\{x\})| > 1$ and we can conclude using step 1) above.
\end{proof}
\begin{remark}
	(i) The drift measures $\hat{\nu}$ and $\nu$ as defined in Proposition \ref{prop:equiv_mvd_right_and_sym_loc_time} (i) and (ii), 
	respectively, are of course not unique.
	
	(ii) Proposition \ref{prop:equiv_mvd_right_and_sym_loc_time}(ii) gives an alternative to conclude the results of Theorem 
	\ref{theorem:reflection_and_absorbing} and its Corollary \ref{cor:reflection_and_absorbing} and \ref{cor:non_existence_of_a_solution} for Eq. 
	(\ref{eqn:SDE_mvd}) with symmetric local time and a drift measure satisfying $\nu(\{x\}) \neq -1$, $x \in \mathbb{R}$, from the results for the 
	right local time.
		
	(iii) The case $\hat{\nu}(\{x\}) = -1$ for some $x \in \mathbb{R}$ is excluded since this condition is responsible for reflection to the left, but 
	this cannot occur for solutions of Eq. (\ref{eqn:SDE_mvd}) where the right local time is chosen: By Lemma \ref{lemma:loc_time_zero} we then have 
	$L_+^X(t,x) = 0$. \mydia
\end{remark}
Similar conclusions as in Proposition \ref{prop:equiv_mvd_right_and_sym_loc_time} can be made when the left local time is involved. For the sake of completeness we state the corresponding results which can be proven by analogous arguments as before.
\begin{proposition}
	(i) $(X,\mathbb{F})$ is a solution of Eq. (\ref{eqn:SDE_mvd}) with right local time and drift measure $\nu$ satisfying 
			$\nu(\{x\}) \neq 1/2$, $x \in \mathbb{R}$, if and only if it is a solution of Eq. (\ref{eqn:SDE_mvd}) with left local time and drift measure
			\[
		 		\overline{\nu}(\myd y) = \left( \frac{1}{1-2\nu(\{y\})}\, \ind_{\{z\in\mathbb{R} : \, \nu(\{z\}) < 1/2\}}(y) 
		 														-  \ind_{\{z\in\mathbb{R} : \, \nu(\{z\}) > 1/2\}}(y)\right) \nu(\myd y)\,.
			\]

	(ii) $(X,\mathbb{F})$ is a solution of Eq. (\ref{eqn:SDE_mvd}) with left local time and drift measure 
			 $\overline{\nu}$ satisfying $\overline{\nu}(\{x\}) \neq -1/2$, $x \in \mathbb{R}$, if and only if it is a solution of Eq. (\ref{eqn:SDE_mvd}) 
			 with right local time and drift measure 
			 \[
		 	 	\nu(\myd y) = \left( \frac{1}{1+2\overline{\nu}(\{y\})}\, \ind_{\{z\in\mathbb{R} : \, \overline{\nu}(\{z\}) > -1/2\}}(y) 
		 					+ \ind_{\{z\in\mathbb{R} : \, \overline{\nu}(\{z\}) < -1/2\}}(y)\right) \overline{\nu}(\myd y)\,.
			 \]
(iii) $(X,\mathbb{F})$ is a solution of Eq. (\ref{eqn:SDE_mvd}) with left local time and drift measure $\overline{\nu}$ if and only 
			if it is a solution of Eq. (\ref{eqn:SDE_mvd}) with symmetric local time and drift measure
			\[
		 		\hat{\nu}(\myd y) = \left( \frac{1}{2+2\overline{\nu}(\{y\})}\, \ind_{\{z\in\mathbb{R} : \, \overline{\nu}(\{z\}) > -1\}}(y) 
		 														+  \ind_{\{z\in\mathbb{R} : \, \overline{\nu}(\{z\}) \leq -1\}}(y)\right) 2\overline{\nu}(\myd y)\,.
			\]

	(iv) $(X,\mathbb{F})$ is a solution of Eq. (\ref{eqn:SDE_mvd}) with symmetric local time and drift measure 
			 $\hat{\nu}$ satisfying $\hat{\nu}(\{x\}) \neq 1$, $x \in \mathbb{R}$, if and only if it is a solution of Eq. (\ref{eqn:SDE_mvd}) 
			 with left local time and drift measure 
			 \[
		 	 	\overline{\nu}(\myd y) = \left( \frac{2}{1-\hat{\nu}(\{y\})}\, \ind_{\{z\in\mathbb{R} : \, \hat{\nu}(\{z\}) < 1\}}(y) 
		 					- \ind_{\{z\in\mathbb{R} : \, \hat{\nu}(\{z\}) > 1\}}(y)\right) \frac{1}{2}\hat{\nu}(\myd y)\,.
			 \]
\end{proposition}
\bibliographystyle{plain}
\bibliography{references}
\end{document}